\documentclass[12pt,a4paper]{article}

\usepackage{amsmath, amsthm, amsfonts, amssymb,color,geometry,enumerate}
\usepackage{graphicx}
\usepackage{fancybox}
\usepackage{cases}
\usepackage{fancyhdr}

\usepackage{amscd} 
\usepackage[all]{xy}

\theoremstyle{definition}

\newtheorem{theorem}{Theorem}[section]
\newtheorem{proposition}[theorem]{Proposition}
\newtheorem{lemma}[theorem]{Lemma}

\newtheorem{remark}[theorem]{Remark}
\newtheorem{problem}[theorem]{Problem}
\newtheorem{example}[theorem]{Example}

\newtheorem{corollary}[theorem]{Corollary}

\makeatletter

\@addtoreset{equation}{section}
\makeatother

\renewcommand{\epsilon}{\varepsilon}    % epsilon 
\begin{document}

\title{Max-convolution semigroups and extreme values in limit theorems for the free multiplicative convolution}
\author{\Large{Yuki Ueda}}
\date{}

\maketitle

\abstract{We investigate relations between additive convolution semigroup and max-convolution semigroup through the law of large numbers for the free multiplicative convolution. Based on the relation, we give a formula related with Belinschi-Nica semigroup and max-Belinschi-Nica semigroup. Finally, we give several limit theorems for classical, free and Boolean extreme values.}

%\tableofcontents

%%%%%%%%%%%%%%%%%%%%%%%%%%%%
%%%%%%%%%%%%%%%%%%%%%%%%%%%%

\section{Introduction}

Denote by $\mathcal{P}$ and $\mathcal{P}_+$ the set of all probability measures on $\mathbb{R}$ and $[0,\infty)$, respectively. In classical probability theory, many mathematicians studied the limit law of addition and multiplication of large numbers of independent random variables. In free probability theory, we obtain the limit law of addition of large numbers of freely independent real random variables (selfadjoint operators): for any $\mu\in\mathcal{P}$ with the first moment $\alpha$, we have $D_{1/n}(\mu^{\boxplus n})\xrightarrow{w} \delta_\alpha$ as $n\rightarrow\infty$  (see \cite[Corollary 5.2]{LP97}), where $\boxplus$ is called {\it free additive convolution} and
\begin{align*}
\mu^{\boxplus n}:=\overbrace{\mu\boxplus \cdots \boxplus \mu}^{n \text{ times}}
\end{align*}
is the {\it free additive convolution powers of $\mu$} (see \cite{V86,M92,BV93}) and $D_c$ is the {\it dilation} i.e. $D_c(\mu)(B):=\mu(c^{-1}B)$ for all $c>0$ and Borel sets $B$ in $\mathbb{R}$. Similarly, Tucci studied the limit law of multiplication of large numbers of freely independent bounded positive random variables (see \cite{T10}). After that, Haagerup and M\"{o}ller extended Tucci's limit theorem to (unbounded) positive random variables: for any $\mu\in\mathcal{P}_+$, there exists a unique $\nu\in\mathcal{P}_+$ such that $(\mu^{\boxtimes n})^{1/n}\xrightarrow{w} \nu$ as $n\rightarrow\infty$ (see \cite{HM13}), where $\boxtimes$ is called {\it free multiplicative convolution} and 
\begin{align*}
\mu^{\boxtimes n}:=\overbrace{\mu\boxtimes \cdots \boxtimes \mu}^{n \text{ times}}\end{align*}
is the {\it free multiplicative convolution powers of $\mu$} (see \cite{V87,BV93}) and $\mu^{1/n}$ is the distribution of $X^{1/n}$ if $X\sim \mu$. Denoted by $\Phi(\mu)$ the weak limit law of $(\mu^{\boxtimes n})^{1/n}$ as $n\rightarrow\infty$. 

In \cite{BV06}, Ben Arous and Voiculescu introduced the {\it free max-convolution} ${\Box \hspace{-.75em} \lor}$. This convolution is the distribution of maximum of free independent real random variables with respect to the spectral order (which was introduced in \cite{O71}). One of the most important distributions in free max-probability is the {\it free extreme value distribution}. This distribution has many similarity to the classical extreme value distribution which is the limit law of the maximum of large numbers of identically distributed random variables (see e.g. \cite{R87}). In \cite{BC10}, we constructed random matrix models which realizes free extreme value distribution. In \cite{GN17}, we obtained relations between free extreme values, order statistics (e.g. Peak-Over-Threshold method and generalized extreme values) and random matrices.

In \cite{VV18}, Vargas and Voiculescu introduced the {\it Boolean max-convolution} ${\cup \hspace{-.72em} \lor}$. This convolution is the distribution of maximum of Boolean independent random variables with respect to spectral order. The maximum works for Boolean independent nonnegative random variables. In the same way as free max-case, we obtain extreme value distribution with respect to Boolean max-convolution and it corresponds to the {\it Dagum distribution}.

In Section 3.1, we give a formula between free additive convolution $\boxplus$ and free max-convolution ${\Box \hspace{-.75em} \lor}$ through using the above operator $\Phi:\mathcal{P}_+\rightarrow\mathcal{P}_+$.

\begin{theorem}\label{thm:free_additive_max}
Consider $\mu\in \mathcal{P}_+$. Then we have
\begin{align*}
\Phi(D_{1/t}(\mu^{\boxplus t}))=\Phi(\mu)^{\Box \hspace{-.55em} \lor t}, \qquad t\ge 1.
\end{align*}
\end{theorem}
In Section 3.2, we claim that the operator $\Phi$ connects {\it Boolean additive convolution} $\uplus$ (see \cite{SW97}) to Boolean max-convolution ${\cup \hspace{-.67em}\lor}$ as follows.
\begin{theorem}\label{thm:Boolean_additive_max}
Consider $\mu\in \mathcal{P}_+$. Then we have
\begin{align*}
\Phi(D_{1/t}(\mu^{\uplus t}))=\Phi(\mu)^{\cup \hspace{-.52em} \lor t}, \qquad t>0.
\end{align*}
\end{theorem}

Next, we define two operators $B_t$ and $B_t^\lor$ as follows:
\begin{align*}
B_t(\mu)&:=(\mu^{\boxplus (1+t)})^{\uplus \frac{1}{1+t}}, \qquad t\ge 0, \hspace{2mm} \mu\in\mathcal{P};\\
B_t^\lor(\mu)&:=(\mu^{\Box \hspace{-.55em} \lor (1+t)})^{\cup \hspace{-.52em} \lor \frac{1}{1+t}}, \qquad t\ge 0, \hspace{2mm} \mu\in\mathcal{P}_+.
\end{align*}
It is known that $B_t\circ B_s=B_{t+s}$ and $B_t^\lor \circ B_s^\lor=B_{t+s}^\lor$ for all $t,s\ge 0$, so that the families $\{B_t\}_{t\ge 0}$ and $\{B_t^\lor\}_{t\ge 0}$ are semigroups with respect to the composition of operators. The semigroups $\{B_t\}_{t\ge0}$ and $\{B_t^\lor\}_{t\ge0}$ are called {\it Belinschi-Nica semigroup} (see \cite{BN08}) and {\it max-Belinschi-Nica semigroup} (see \cite{U19}), respectively. Considering these semigroups is important to understand relations between free and Boolean type limit theorems or free-max and Boolean-max type limit theorems. In Section 3.3, by using the operator $\Phi$, we claim that Belinschi-Nica semigroup $\{B_t\}_{t\ge0}$ and max-Belinschi-Nica semigroup $\{B_t^\lor\}_{t\ge0}$ are closely intertwined with each other.

\begin{theorem}\label{thm:intertwiner}
Consider $\mu\in \mathcal{P}_+$. Then we have
\begin{align*}
\Phi\circ B_t(\mu)=B_t^\lor\circ \Phi(\mu), \qquad t\ge 0.
\end{align*}
\end{theorem}

In Section 3.4, we construct an operator which connects classical additive convolution $\ast$ to classical max-convolution $\lor$. A probability measure $\mu$ is said to be (classical) {\it infinitely divisible} if for each $n\in\mathbb{N}$ there is $\mu_n\in\mathcal{P}$ such that $\mu=\mu_n^{\ast n}$. Let $\text{ID}_+$ be the set of all (classical) {\it infinitely divisible distributions} on $[0,\infty)$ (for details of infinitely divisible distributions, see e.g. \cite{Sato}). An operator $\Psi:\text{ID}_+\rightarrow\mathcal{P}_+$ is defined by
\begin{align*}
\Psi:=\mathcal{X}^{\lor} \circ \Phi \circ \mathcal{X}^{-1},
\end{align*}
where  the operator $\mathcal{X}$ is defined by 
\begin{align*}
\mathcal{X}:=\Lambda^{-1}\circ B_1,
\end{align*}
and $\Lambda$ is the {\it Bercovici-Pata bijection} (see \cite{BP99}). Note that the operator $\mathcal{X}$ is a bijection from $\mathcal{P}_+$ to $\text{ID}_+$ and it is called the {\it Boolean-classical Bercovici-Pata bijection} (see \cite{BP99,BN08}). Moreover, for any $\mu\in\mathcal{P}_+$, the measure $\mathcal{X}^{\lor}(\mu)$ is characterized by
\begin{align*}
\mathcal{X}^{\lor}(\mu)([0,\cdot]):=\exp\left[ 1-\frac{1}{\mu([0,\cdot])}\right].
\end{align*}
The operator $\mathcal{X}^\lor$ is called the {\it Boolean-classical max-Bercovici-Pata bijection} which was firstly introduced by \cite{VV18}. Then we obtain the following formula.

\begin{theorem}\label{classical_additive_max2}
Consider $\mu\in \text{ID}_+$. Then we have
\begin{align*}
\Psi(D_{1/t}(\mu^{\ast t}))=\Psi(\mu)^{\lor t}, \qquad t>0.
\end{align*}
\end{theorem}

In Section 4, we compute some probability measures in the classes $\Phi(\mathcal{P}_+)$ and $\Psi(\mathcal{P}_+)$. As one of the most important computations, we mention that the operator $\Phi$ connects the free/Boolean stable laws to the free/Boolean extreme values, respectively. As one more, we claim that the operator $\Psi$ maps the classical stable laws to the classical extreme values. 

In Section 5, we give a few of limit theorems for the free and Boolean extreme values by using limit theorems for free multiplicative convolution. Through discussions at Section 5, we mention that Marchenko-Pastur law is closely related with the free and Boolean extreme values.

%%%%%%%%%%%%%%%
%PRELIMINARIES
%%%%%%%%%%%%%%%

\section{Preliminaries}

%%%%%%%%%%%%%%%%%%%%%%%%%%%%%%
%ATOMS OF FREE AND Boolean CONVOLUTIONS
%%%%%%%%%%%%%%%%%%%%%%%%%%%%%%

\subsection{Atoms of free and Boolean additive convolutions}

Consider freely independent noncommutative real random variables $X\sim \mu$ and $Y\sim \nu$. Then we define $\mu\boxplus\nu$ as the distribution of $X+Y$ and the operation $\boxplus$ is called the {\it free additive convolution} which was introduced by \cite{V86} (see also \cite{M92, BV93}). We can define the partial semigroup $\{\mu^{\boxplus t}\}_{t\ge 1}$ with $\mu^{\boxplus 1}=\mu$ with respect to free convolution, for all $\mu\in\mathcal{P}$. In \cite{BB04}, we get a location of an atom of $\mu^{\boxplus t}$.

\begin{lemma}\label{lem:atom} \cite{BB04}
Consider $\mu\in\mathcal{P}$, $t>1$ and $\alpha\in\mathbb{R}$. Then $\mu^{\boxplus t}$ has an atom $\alpha$ if and only if $\mu(\{\alpha/t\})> 1-t^{-1}$. In the case, we have
\begin{align*}
\mu^{\boxplus t}(\{\alpha\})=t\mu\left(\left\{\frac{\alpha}{t}\right\}\right)-(t-1).
\end{align*}
\end{lemma}

For $\mu\in\mathcal{P}$, we define the following functions:
\begin{align*}
G_\mu(z):=\int_\mathbb{R} \frac{1}{z-x}d\mu(x), \qquad F_\mu(z):=\frac{1}{G_\mu(z)}.
\end{align*}
The function $G_\mu$ is called the {\it Cauchy transform of $\mu$} and it is analytic on the upper complex plane $\mathbb{C}^+$ taking values in the below complex plane $\mathbb{C}^-$.
We know a useful criterion for locating an atom $\alpha$ of $\mu$.
\begin{lemma}\label{lem:F}
Consider $\mu\in\mathcal{P}$ and $\alpha\in\mathbb{R}$. Then $\mu$ has an atom $\alpha$ if and only if $F_\mu(\alpha)=0$ and the {\it Juria-Carath\'{e}odory derivative} $F_\mu'(\alpha)$ is finite where $F_\mu'(\alpha)$ is the limit of
\begin{align*}
\frac{F_\mu(z)-F_\mu(\alpha)}{z-\alpha},
\end{align*}
as $z\rightarrow\alpha$ nontangentially and $z\in\mathbb{C}^+$. In this case, we have $\mu(\{\alpha\})=F_\mu'(\alpha)^{-1}$.
\end{lemma}

Consider Boolean independent real random variables $X\sim \mu$ and $Y\sim \nu$. Then we define $\mu\uplus \nu$ as the distribution of $X+Y$ and the operation $\uplus$ is called the {\it Boolean additive convolution} which was introduced by \cite{SW97}. The Boolean additive convolution is characterized by the {\it self-energy function} which is defined by
\begin{align*}
E_\mu(z):=z-F_\mu(z), \qquad z\in\mathbb{C}^+
\end{align*}
for all $\mu\in\mathcal{P}$, that is, $E_{\mu\uplus \nu}=E_\mu+E_\nu$ for all $\mu,\nu\in\mathcal{P}$. In \cite{SW97}, for all $\mu\in\mathcal{P}$ and $t>0$, there exists a unique $\mu_t\in\mathcal{P}$ such that $E_{\mu_t}=tE_\mu$. A family $\{\mu^{\uplus t}\}_{t\ge 0}$ is a semigroup with respect to the Boolean additive convolution such that $\mu_0=\delta_0$. Write $\mu^{\uplus t}:=\mu_t$ for all $\mu\in\mathcal{P}$ and $t\ge 0$. Therefore

\begin{align}\label{F}
F_{\mu^{\uplus t}}(z)=(1-t)z+tF_\mu(z), \qquad z\in\mathbb{C}^+.
\end{align}

%\begin{proposition}\label{prop:Boolean}\cite{SW97}
%Consider $\mu,\nu\in\mathcal{P}$. Then we have
%\begin{align*}
%E_{\mu\uplus \nu}=E_\mu+E_\nu, \qquad E_{\mu^{\uplus t}}=tE_\mu.
%\end{align*}
%\end{proposition}

According to the above discussion, we obtain the following implication.
\begin{corollary}\label{cor:Boolean}
Consider $\mu\in\mathcal{P}$ not being $\delta_0$. Then $\mu(\{0\})=0$ if and only if $\mu^{\uplus t}(\{0\})=0$ for some (any) $t>0$.
\end{corollary}
\begin{proof}
Consider firstly $\mu(\{0\})=0$. Assume that $\mu^{\uplus t}(\{0\})\neq 0$ for any (some) $t>0$. Then $F_{\mu^{\uplus t}}(0)=0$ and $F_{\mu^{\uplus t}}'(0)<\infty$ by Lemma \ref{lem:F}. By the equation \eqref{F}, we have $F_\mu(0)=0$ and 
\begin{align*}
F_\mu'(0)=-\frac{1-t}{t}+\frac{1}{t}F_{\mu^{\uplus t}}'(0)<\infty.
\end{align*}
By Lemma \ref{lem:F} again, we have $\mu(\{0\})\neq 0$. This is a contradiction for $\mu(\{0\})=0$. The converse implication is proved by the same way.
\end{proof}

Consequencely, we have that  $\mu$ has an atom $0$ if and only if $\mu^{\uplus t}$ also has an atom $0$. In this case, we have
\begin{align}\label{eq:Boolean_add_atom}
\mu^{\uplus t}(\{0\})=\frac{\mu(\{0\})}{t-(t-1)\mu(\{0\})}, \qquad t>0.
\end{align}

%%%%%%%%%%%%%%%%%%%%%%%%%%
%HAAGERUP-MOLLER TYPE LIMIT THEOREM
%%%%%%%%%%%%%%%%%%%%%%%%%%

\subsection{Limit theorem for free multiplicative convolution}

For probability measures $\mu\in\mathcal{P}_+$ and $\nu\in\mathcal{P}$, we write $\mu\boxtimes \nu\in\mathcal{P}$ as the distribution of $\sqrt{X}Y\sqrt{X}$, where $X\ge0$ and $Y$ are freely independent random variables distributed as $\mu$ and $\nu$, respectively. The operation $\boxtimes$ is called {\it free multiplicative convolution}. This was firstly introduced in 
\cite{V87} as the distribution of multiplication of bounded random variables. Finally, it was extended to unbounded random variables (see \cite{BV93}).

Consider $\mu\in\mathcal{P}_+$ with $\mu\neq \delta_0$. We define
\begin{align*}
\Psi_\mu(x):=\int_0^\infty \frac{tx}{1-tx}d\mu(t).
\end{align*}
Then its inverse function (namely $\Psi^{-1}_\mu$) exists in a neighborhood of $(\mu(\{0\})-1,0)$. We define the {\it S-transform of $\mu$} by setting
\begin{align*}
S_\mu(z)=\frac{z+1}{z}\Psi^{-1}_\mu(z), \qquad z\in (\mu(\{0\})-1,0).
\end{align*}
 In \cite[Corollary 6.6]{BV93}, for probability measures $\mu,\nu\neq \delta_0$ on $[0,\infty)$, we have $S_{\mu\boxtimes \nu}=S_\mu S_\nu$ on the common interval defined three S-transforms. In \cite{B03}, we know an atom of $\mu\boxtimes \nu$ at $0$ as follows:
 \begin{align}\label{atom:freemulti}
 (\mu\boxtimes\nu)(\{0\})=\max\{\mu(\{0\}),\nu(\{0\})\}.
 \end{align}
 Therefore we get $\mu^{\boxtimes n}(\{0\})=\mu(\{0\})$ for all $n\in\mathbb{N}$ by induction. Moreover, we give formulas of the S-transform with respect to free and Boolean additive convolutions. 
 
\begin{lemma}\label{lem:Sfb} \cite{BN08}
For $\mu\in\mathcal{P}_+$ with $\mu\neq \delta_0$, we have
\begin{align*}
S_{\mu^{\boxplus t}}(z)&=\frac{1}{t}S_\mu\left(\frac{z}{t}\right), \qquad t\ge 1.\\
S_{\mu^{\uplus t}}(z)&=\frac{1}{t}S_\mu\left( \frac{z}{t-z+tz}\right), \qquad t>0.
\end{align*}
\end{lemma}

We give a functional property of the S-transform.

\begin{lemma}\label{lem:image}(see \cite[Theorem 4.4]{HL00} and \cite{HM13})
Consider $\mu\in\mathcal{P}_+$ not being a Dirac measure. Then $S_\mu$ is strictly decreasing on $(\mu(\{0\})-1,0)$. Moreover, we have $S_\mu((\mu(\{0\})-1,0))=(b_\mu^{-1},a_\mu^{-1})$, where $0\le a_\mu<b_\mu\le \infty$ are defined by
\begin{align}\label{ab}
a_\mu:=\left( \int_0^\infty x^{-1} d\mu(x)\right)^{-1}, \qquad b_\mu:=\int_0^\infty xd\mu(x).
\end{align}
Note that if $\mu(\{0\})>0$, then we understand $a_\mu^{-1}=\infty$.
\end{lemma}

Tucci (\cite{T10}), Haagerup and M\"{o}ller (\cite{HM13}) give the following limit theorem for free multiplicative convolution.

\begin{proposition}\label{prop:HMlimit}\cite{T10, HM13}
Consider $\mu\in\mathcal{P}_+$. A sequence of probability measures $(\mu^{\boxtimes n})^{1/n}$, weakly converges to some probability measure (namely, $\Phi(\mu)$) on $[0,\infty)$.  If $\mu$ is a Dirac measure on $[0,\infty)$, then $\Phi(\mu)=\mu$. If $\mu$ is not a Dirac measure on $[0,\infty)$, then $\Phi(\mu)$ is uniquely determined and it satisfies
\begin{align*}
\Phi(\mu)(\{0\})=\mu(\{0\}), \qquad \Phi(\mu)\left(\left[0,\frac{1}{S_\mu(x-1)} \right]\right)=x
\end{align*}
for any $x\in(\mu(\{0\}),1)$. The support of $\Phi(\mu)$ is the closure of $(a_\mu,b_\mu)$.
\end{proposition}

We prove that the operator $\Phi$ commutes with the dilation.
\begin{lemma}\label{lem:dP}
For all $c>0$, we have $D_c\circ \Phi=\Phi\circ D_c$ on $\mathcal{P}_+$.
\end{lemma}
\begin{proof}
For all $\mu\in\mathcal{P}_+$ and for all $n\in\mathbb{N}$, we have
\begin{align*}
D_c((\mu^{\boxtimes n})^{1/n})=(D_{c^n}(\mu^{\boxtimes n}))^{1/n}=(D_c(\mu)^{\boxtimes n})^{1/n}.
\end{align*}
As $n\rightarrow\infty$, we obtain $D_c\circ \Phi(\mu)=\Phi\circ D_c(\mu)$.
\end{proof}

We give a relation between $\Phi$ and S-transform as follows.
\begin{lemma}\label{lem:pS}
For all $\mu\in \mathcal{P}_+$ not being a Dirac measure and for all $x\in (a_\mu,b_\mu)$, we have
\begin{align*}
\Phi(\mu)([0,x])=S_\mu^{-1}\left(\frac{1}{x}\right)+1,
\end{align*}
and therefore $(a_\mu,b_\mu)=\{x:\Phi(\mu)([0,x])\in (\mu(\{0\}),1)\}$.
\end{lemma}
\begin{proof}
For all $x\in (a_\mu, b_\mu)$ we have $S_\mu^{-1}(1/x)\in (\mu(\{0\})-1,0)$ by Lemma \ref{lem:image}. Since 
\begin{align*}
x=\frac{1}{S_\mu(S_\mu^{-1}(1/x)+1-1)},
\end{align*}
we have
\begin{align*}
\Phi(\mu)([0,x])=\Phi(\mu)\left( \left[0,  \frac{1}{S_\mu(S_\mu^{-1}(1/x)+1-1)}\right]\right)=S_\mu^{-1}\left(\frac{1}{x}\right)+1,
\end{align*}
for all $x\in (a_\mu,b_\mu)$.
\end{proof}

\begin{remark}\label{rem:support}
By Proposition \ref{prop:HMlimit} and Lemma \ref{lem:pS}, the support of $\Phi(\mu)$ is the closure of $\{x:\Phi(\mu)([0,x])\in (\mu(\{0\}),1)\}$.
\end{remark}

%%%%%%%%%%%%%%
%MAX CONVOLUTIONS
%%%%%%%%%%%%%%

\subsection{Max-convolutions}

%%%%%%%%%%%%%%
%CLASSICAL MAX CONVOLUTIONS
%%%%%%%%%%%%%%

\subsubsection{Classical max-convolution}

Let $(\Omega, \mathcal{F},\mathbb{P})$ be a classical probability space. For an ($\mathcal{F}$-measurable) real random variable $X$, we define a distribution function $F_X$ of $X$, that is, $F_X(\cdot):=\mathbb{P}(X\le \cdot)$. For independent real random variables $X$ and $Y$, we have 
\begin{align*}
F_{X\lor Y}(x)&=\mathbb{P}(X\lor Y\le x)=\mathbb{P}(X\le x, Y\le x)\\
&=\mathbb{P}(X\le x)\mathbb{P}(Y\le x)=F_X(x)F_Y(x), \qquad x\in\mathbb{R},
\end{align*}
where $X\lor Y:=\max\{X,Y\}$. According to the above calculation, we define the {\it classical max-convolution} $\mu\lor \nu$ of $\mu,\nu\in\mathcal{P}$ as 
\begin{align*}
\mu\lor \nu ((-\infty,\cdot]):=\mu((-\infty, \cdot])\nu((-\infty,\cdot]).
\end{align*}

For $n\in\mathbb{N}$ and $\mu\in\mathcal{P}$, we define $\mu^{\lor n}:=\overbrace{\mu\lor\cdots \lor\mu}^{n\text{ times}}$. More generally, for $t>0$, we define
\begin{align*}
\mu^{\lor t}((-\infty,\cdot]):=\mu((-\infty,\cdot])^t.
\end{align*}

A non-trivial distribution function $F$ is said to be {\it max-stable} if for any $n\in\mathbb{N}$, there exist $a_n>0$ and $b\in\mathbb{R}$ such that
\begin{align*}
F^{\lor n}(a_n\cdot+b_n)\xrightarrow{w} F(\cdot), \qquad n\rightarrow\infty,
\end{align*}
where $\xrightarrow{w}$ means the convergence at every point of continuity of $F$. The max-stable distributions are characterized as follows.

\begin{proposition}\cite{FT28,F27,G43} $F$ is max-stable if and only if there exist $a>0$ and $b\in\mathbb{R}$ such that $F(ax+b)$ is one of the following distributions:
\begin{align*}
C_{\text{I}}(x):&=\exp(-\exp(-x)) \hspace{3mm}\text{ ({\it Gumbel distribution})};\\
C_{\text{II},\alpha}(x):&=\exp(-x^{-\alpha}) \text{ for } x>0 \text{ and }\alpha>0  \hspace{3mm}\text{ ({\it Fr\'{e}chet distribution})};\\
C_{\text{III},\alpha}(x):&=\exp(-(-x)^\alpha) \text{ for } x\le 0  \text{ and } \alpha>0 \hspace{3mm}\text{ ({\it Weibull distribution})}.
\end{align*}
\end{proposition}

The above distributions are called {\it extreme value distributions}. In mathematical statistics, extreme value distributions are often used to analyze statistical data of rare phenomenon.

%%%
%FREE MAX-CONVOLUTION
%%%

\subsubsection{Free max-convolution}

In free probability theory, the max-convolution was also introduced in \cite{BV06}. Let $(\mathcal{M},\tau)$ be a tracial $W^*$-probability space, that is, $\mathcal{M}$ is a von Neumann algebra and $\tau$ is a normal faithful tracial state on $\mathcal{M}$. We may assume that $\mathcal{M}$ acts on a Hilbert space $\mathcal{H}$. Denoted by $\text{Proj}(\mathcal{M})$ the set of all projections in $\mathcal{M}$ and denoted by $\mathcal{M}_{sa}$ the set of all selfadjoint operators in $\mathcal{M}$. For $P,Q\in\text{Proj}(\mathcal{M})$, we define  $P\lor Q$ as the selfadjoint operator onto $(P\lor Q)\mathcal{H}:=\text{cl}(P\mathcal{H}+Q\mathcal{H})$. Then $P\lor Q\in\text{Proj}(\mathcal{M})$ and it is the maximum of $P$ and $Q$ with respect to the usual operator order. However, it is not necessary that there is the maximum of selfadjoint operators in $\mathcal{M}= \mathcal{B}(\mathcal{H})$ with respect to the operator order (see \cite{K51}). Instead of the operator order, Olson \cite{O71} (see also \cite{A89, BV06}) introduced the spectral order to define the maximum of (bounded) selfadjoint operators on $\mathcal{H}$. In \cite{BV06}, Ben Arous and Voiculescu extended the spectral order to general von Neumann algebras as follows. For $X,Y\in\mathcal{M}_{sa}$, we define $X\prec Y$ by
\begin{align*}
E_X((x,\infty))\le E_Y((x,\infty)), \qquad x\in\mathbb{R},
\end{align*}
where $E_X$ is the spectral projection of $X$ and $\le$ is the usual operator order. The order $\prec$ is called the {\it spectral order}. For any $X,Y\in\mathcal{M}_{sa}$, we define $X\lor Y$ by
\begin{align*}
E_{X\lor Y}((x,\infty)):=E_X((x,\infty))\lor E_Y((x,\infty)), \qquad x\in\mathbb{R}.
\end{align*}
Then $X\lor Y$ is well-defined and $X\lor Y\in\mathcal{M}_{sa}$. Moreover $X\lor Y$ is the maximum of $X$ and $Y$ with respect to the spectral order. Finally, Ben Arous and Voiculescu extended the spectral order to the set of all (unbounded) selfadjoint operators affiliated with $\mathcal{M}$. A (unbounded) selfadjoint operator $X$ on $\mathcal{H}$ is said to be {\it affiliated with $\mathcal{M}$} if $f(X)\in\mathcal{M}$ for all bounded Borel functions $f$ on $\mathbb{R}$, where $f(X)$ is a measurable functional calculus of $X$ with respect to $f$. Note that $X$ is a bounded selfadjoint operator affiliated with $\mathcal{M}$ if and only if $X\in\mathcal{M}$. For any Borel sets $B$ in $\mathbb{R}$, if $f=\mathbf{1}_{B}$, then $f(X)=E_X(B)$. For a selfadjoint operator $X$ affiliated with $\mathcal{M}$, we define a (spectral) distribution function by
\begin{align*}
F_X(x):=\tau(E_X((-\infty,x])), \qquad x\in\mathbb{R}.
\end{align*}
If $X\sim \mu$, then we have $F_X(x)=\mu((-\infty,x])$ for all $x\in\mathbb{R}$.

\begin{proposition}\cite{BV06}
Let $X,Y$ be freely independent real random variables (selfadjoint operators) affiliated with $\mathcal{M}$. Then we have $F_{X\lor Y}=\max\{0,F_X+F_Y-1\}$.
\end{proposition}
For any distribution functions $F,G$ on $\mathbb{R}$, we define
\begin{align*}
F {\Box \hspace{-.75em} \lor} G:=\max\{0,F+G-1\}
\end{align*}
The operation ${\Box \hspace{-.75em} \lor}$ is called the {\it free max-convolution}. We write $\mu {\Box \hspace{-.75em} \lor} \nu$ as the distribution of the maximum of freely independent real random variables $X\sim \mu$ and $Y\sim \nu$, that is,
\begin{align*}
\mu {\Box \hspace{-.75em} \lor} \nu((-\infty,\cdot]):=\mu((-\infty,\cdot]) {\Box \hspace{-.75em} \lor}\nu((-\infty,\cdot])
\end{align*} 

For $n\in\mathbb{N}$ and $\mu\in\mathcal{P}$, we define $\mu^{\Box \hspace{-.55em} \lor n} :=\overbrace{\mu{\Box \hspace{-.75em} \lor} \cdots {\Box \hspace{-.75em} \lor} \mu}^{n \text{ times}}$. More generally, for $t\ge 1$, we define
\begin{align}\label{eq:free-max}
\mu^{\Box \hspace{-.55em} \lor t} ((-\infty,\cdot]):=\max\{t\mu((-\infty,\cdot])-(t-1),0\}.
\end{align}
For $\mu\in\mathcal{P}_+$, we get an atom of $\mu^{\Box \hspace{-.55em} \lor t}$ at $0$ as follows.
\begin{lemma}\label{lem:free-max-atom}
Consider $t> 1$ and $\mu\in\mathcal{P}_+$. Then $\mu^{\Box \hspace{-.55em} \lor t}$ has an atom $0$ if and only if $\mu(\{0\})>1-t^{-1}$. In the case, we have
\begin{align*}
\mu^{\Box \hspace{-.55em} \lor t}(\{0\})=t\mu(\{0\})-(t-1).
\end{align*}
\end{lemma}
\begin{proof}
By definition \eqref{eq:free-max}, we have
\begin{align*}
\mu^{\Box \hspace{-.55em} \lor t} (\{0\})=\max\{t\mu(\{0\})-(t-1),0\}.
\end{align*}
Therefore $\mu^{\Box \hspace{-.55em} \lor t}$ has an atom $0$ if and only if $t\mu(\{0\})-(t-1)>0$, i.e., $\mu(\{0\})>1-t^{-1}$.
\end{proof}

A non-trivial distribution function $F$ is said to be {\it freely max-stable} if for any $n\in\mathbb{N}$, there exists $a_n>0$ and $b_n\in\mathbb{R}$ such that
\begin{align*}
F^{\Box \hspace{-.55em} \lor n} (a_n \cdot+b_n) \xrightarrow{w} F(\cdot), \qquad n\rightarrow\infty,
\end{align*}
The freely max-stable distributions are characterized as follows.

\begin{proposition} \cite[Theorem 6.8]{BV06}
$F$ is freely max-stable if and only if there exist $a>0$ and $b\in\mathbb{R}$ such that $F(ax+b)$ is one of the following distributions:
\begin{align*}
F_{\text{I}}(x):&=\max\{0, 1-e^{-x}\} \hspace{3mm}\text{ ({\it Exponential distribution})};\\
F_{\text{II},\alpha}(x):&=\max\{0, 1-x^{-\alpha}\} \text{ for }\alpha>0 \hspace{3mm}\text{ ({\it Pareto distribution})};\\
F_{\text{III},\alpha}(x):&=1-|x|^\alpha \text{ for }  -1\le x\le 0 \text{ and } \alpha>0 \hspace{3mm}\text{ ({\it Beta law})}.
\end{align*}
The above distributions are called {\it free extreme value distributions}. 
\end{proposition}

Define a function $\Lambda^{\lor}$ on $[0,1]$ by setting $\Lambda^{\lor}(0):=0$ and $\Lambda^{\lor}(x):=\max\{0,1+\log x\}$ for $x\in(0,1]$. If $F$ is a distribution function on $\mathbb{R}$, so is $\Lambda^{\lor}(F)$. Note that $\Lambda^{\lor}$ maps the classical extreme values to the corresponding type free extreme values.

For a probability measure $\mu$ on $\mathbb{R}$, we define the probability measure $\Lambda^{\lor}(\mu)$ such that 
\begin{align*}
\Lambda^{\lor}(\mu)((-\infty, \cdot]):=\Lambda^{\lor} (\mu((-\infty,\cdot])).
\end{align*}
It is known that the operator $\Lambda^{\lor}$ is a homomorphism from $(\mathcal{P},\lor)$ to $(\mathcal{P},{\Box \hspace{-.75em} \lor})$, that is,
\begin{align*}
\Lambda^{\lor}(\mu\lor \nu)=\Lambda^{\lor}(\mu){\Box \hspace{-.75em} \lor}\Lambda^{\lor}(\nu),
\end{align*}
for all probability measures $\mu$ and $\nu$ on $\mathbb{R}$.

The function $\Lambda^{\lor}$ is surjective on $[0,1]$, but it is not injective on $[0,1]$. This proof is very simple. Firstly, we show that it is surjective. Let $\Pi^{\lor}$ be a function on $[0,1]$ defined by
\begin{align*}
\Pi^{\lor}(x):=\exp(-(1-x)).
\end{align*}
For all $y\in[0,1]$, we have
\begin{equation}\label{surj of lambda}
\Lambda^{\lor}(\Pi^{\lor}(y))=\max\{0,1+\log(\exp(-(1-y)))\}=y.
\end{equation}
Next, we show that it is not injective. For example, if $x\neq y<e^{-1}$, then $\Lambda^{\lor}(x)=0=\Lambda^{\lor}(y)$. Therefore it is not injective.

For $\mu\in\mathcal{P}$, we define
\begin{align*}
\Pi^{\lor}(\mu)((-\infty,\cdot]):=\Pi^{\lor}(\mu([0,\cdot]))=\exp(-(1-\mu((-\infty,\cdot]))).
\end{align*}
The measure $\Pi^{\lor}(\mu)$ is called the {\it max-compound Poisson law with $\mu$}.

%%%
%Boolean MAX-CONVOLUTION
%%%

\subsubsection{Boolean max-convolution}

Let $\mathcal{H}$ be a Hilbert space and $\xi\in\mathcal{H}$ a unit vector. Define the vector state $\varphi_\xi$ by setting
\begin{align*}
\varphi_\xi(T):=\langle T\xi , \xi\rangle,
\end{align*}
for all selfadjoint operator $T$ on $\mathcal{H}$, where $\langle \cdot, \cdot \rangle$ is the inner product on $\mathcal{H}$. In this section, we give a (spectral) distribution function $F_T$ of selfadjoint operator $T$ on $\mathcal{H}$ by 
\begin{align*}
F_T(x):=\varphi_\xi (E_T((-\infty, x])), \qquad x\in\mathbb{R}.
\end{align*}
\begin{proposition} \cite[Lemma 3.3]{VV18}
Let $X\ge0$, $Y\ge0$ be Boolean independent random variables on $(\mathcal{H},\xi)$. Then we have
\begin{align*}
F_{X\lor Y}=\frac{F_XF_Y}{F_X+F_Y-F_XF_Y}=:F_X {\cup \hspace{-.67em}\lor}F_Y.
\end{align*}
\end{proposition}
We understand $F_X {\cup \hspace{-.67em}\lor}F_Y(x)=0$ when $x\ge0$ satisfies $F_X(x)=0$ or $F_Y(x)=0$. We write $\mu{\cup \hspace{-.67em}\lor}\nu$ as the distribution of the maximum of Boolean independent positive random variables $X\sim\mu\in\mathcal{P}_+$ and $Y\sim\nu\in\mathcal{P}_+$, that is,
\begin{align*}
\mu{\cup \hspace{-.67em}\lor}\nu([0,\cdot]):=\mu([0,\cdot]){\cup \hspace{-.67em}\lor}\nu([0,\cdot]).
\end{align*}
The operation ${\cup \hspace{-.67em}\lor}$ is called the {\it Boolean max-convolution}. For $n\in\mathbb{N}$ and $\mu\in\mathcal{P}_+$, we define $\mu^{\cup \hspace{-.52em} \lor n}:=\overbrace{\mu{\cup \hspace{-.67em} \lor} \cdots {\cup \hspace{-.67em} \lor} \mu}^{n \text{ times}}$. More generally, for $t>0$, we define
\begin{align}\label{eq:Boolean-max}
\mu^{\cup \hspace{-.52em} \lor t} ([0,\cdot]):=\frac{\mu([0,\cdot])}{t-(t-1)\mu([0,\cdot])}.
\end{align}

By definition \eqref{eq:Boolean-max}, we get the following lemma.
\begin{lemma}\label{lem:Boolean-max-atom}
Consider $t>0$ and $\mu\in\mathcal{P}_+$. Then $\mu$ has an atom $0$ if and only if $\mu^{\cup \hspace{-.52em} \lor t}$ also has an atom $0$. In the case, we have
\begin{align*}
\mu^{\cup \hspace{-.52em} \lor t}(\{0\})=\frac{\mu(\{0\})}{t-(t-1)\mu(\{0\})}.
\end{align*}
\end{lemma}

A non-trivial distribution function $F$ on $[0,\infty)$ is said to be {\it Boolean max-stable} if for any $n\in\mathbb{N}$, there exists $a_n>0$ such that
\begin{align*}
F^{\cup \hspace{-.52em} \lor n} (a_n \cdot) \xrightarrow{w} F(\cdot), \qquad n\rightarrow\infty.
\end{align*}
\begin{proposition}\cite[Theorem 6.8]{VV18}
$F$ is Boolean max-stable if and only if there exist $a>0$ and $b\in\mathbb{R}$ such that $F(ax+b)$ is the following distribution:
\begin{align*}
B_{\text{II},\alpha}(x):=(1+x^{-\alpha})^{-1}
\end{align*}
for some $\alpha>0$. This is called the {\it Dagum distribution} or the {\it (type II) Boolean extreme value distribution}.
\end{proposition}

\begin{problem}
We do not know operator models which realize the maximum of Boolean independent general selfadjoint operators. Can we obtain such operator models? Moreover, we should find other type Boolean extreme values.
\end{problem}

Define a function $\mathcal{X}^{\lor}$ on $[0,1]$ by setting $\mathcal{X}^{\lor}(0):=0$ and $\mathcal{X}^{\lor}(x):=\exp(1-x^{-1})$ for $x\in(0,1]$. If $F$ is a distribution function on $\mathbb{R}$, so is $\mathcal{X}^{\lor}(F)$. Note that $\mathcal{X}^{\lor}$ maps the Boolean extreme values to the classical extreme values.

For a probability measure $\mu$ on $\mathbb{R}$, we define the probability measure $\mathcal{X}^{\lor}(\mu)$ such that 
\begin{align*}
\mathcal{X}^{\lor}(\mu)((-\infty, \cdot]):=\mathcal{X}^{\lor} (\mu((-\infty,\cdot])).
\end{align*}
The operator $\mathcal{X}^{\lor}$ is a homomorphism from $(\mathcal{P}_+,{\cup \hspace{-.67em}\lor})$ to $(\mathcal{P}_+,\lor)$, that is,
\begin{align*}
\mathcal{X}^{\lor}(\mu {\cup \hspace{-.67em}\lor} \nu)=\mathcal{X}^{\lor}(\mu)\lor \mathcal{X}^{\lor}(\nu),
\end{align*}
for all probability measures $\mu$ and $\nu$ on $[0,\infty)$. Moreover, it is clear that the function $\mathcal{X}^\lor$ is a bijection and its inverse function is given by 
\begin{align*}
(\mathcal{X}^{\lor})^{-1}(x)&=\frac{1}{1-\log x}, \qquad  x\in(0,1]\\
(\mathcal{X}^{\lor})^{-1}(0)&=0.
\end{align*}
The operator $\mathcal{X}^{\lor}$ is called the {\it Boolean-classical max-Bercovici-Pata bijection} (see \cite{VV18, U19}). 

%%%
%MAX-BELINSCHI-NICA SEMIGROUP
%%%

\subsubsection{Max-Belinschi-Nica semigroup}

We firstly give the {\it Belinschi-Nica semigroup} $\{B_t\}_{t\ge0}$ introduced by \cite{BN08}:
\begin{align*}
B_t(\mu):=(\mu^{\boxplus (1+t)})^{\uplus \frac{1}{1+t}}, \qquad \mu\in\mathcal{P}_+, \hspace{2mm} t\ge0.
\end{align*}
In \cite{BN08}, we have $B_t\circ B_s=B_{t+s}$ for all $t,s\ge0$ and $B_1(\mu\uplus\nu)=B_1(\mu)\boxplus B_1(\nu)$ for all $\mu,\nu\in\mathcal{P}_+$. Moreover, $B_t$ is a homomorphism with respect to $\boxtimes$, that is, $B_t(\mu\boxtimes\nu)=B_t(\mu)\boxtimes B_t(\nu)$ for all $t\ge0$ and $\mu,\nu\in\mathcal{P}_+$. Furthermore, $B_1$ connects Boolean type limit theorem to free type one. In \cite{BP99}, for a sequence $\{\mu_n\}_n$ in $\mathcal{P}$ and $\{k_n\}_n$ in $\mathbb{N}$ with $k_1<k_2<\cdots$ and $k_n\rightarrow\infty$ as $n\rightarrow\infty$, there exists $\mu\in\mathcal{P}$ such that $\mu_n^{\uplus k_n} \xrightarrow{w} \mu$ if and only if there exists a unique $\nu\in\mathcal{P}$ (which is freely infinitely divisible, for short, FID) such that $\mu_n^{\boxplus k_n}\xrightarrow{w} \nu$ as $n\rightarrow\infty$. In \cite{BN08}, we get $\nu=B_1(\mu)$ under the above settings. Note that $B_1$ is a bijection from $\mathcal{P}$ to the set of all FID distributions on $\mathbb{R}$. For example, $B_1$ maps the Boolean stable laws to the corresponding freely stable laws.

Next, we give one parameter family $\{B_t^\lor\}_{t\ge0}$ of operators on $\mathcal{P}_+$ defined by
\begin{align*}
B_t^\lor(\mu):=(\mu^{\Box \hspace{-.55em} \lor (1+t)})^{\cup \hspace{-.52em} \lor \frac{1}{1+t}}, \qquad \mu\in\mathcal{P}_+, \hspace{2mm} t\ge 0,
\end{align*}
where the family was introduced by \cite{U19}. We know that $B_t^\lor\circ B_s^\lor=B_{t+s}^\lor$ for all $t,s\ge0$. The family $\{B_t^\lor\}_{t\ge0}$ is said the {\it max-Belinschi-Nica semigroup}. The semigroup is very similar to the original Belinschi-Nica semigroup. For example, $B_1^\lor$ is a homomorphism from $(\mathcal{P}_+,{\cup \hspace{-.67em}\lor})$ to $(\mathcal{P}_+,{\Box \hspace{-.75em}\lor})$, that is,
\begin{align*}
B_1^\lor(\mu{\cup \hspace{-.67em}\lor}\nu)=B_1^\lor(\mu) {\Box \hspace{-.72em} \lor}B_1^\lor(\nu),
\end{align*}
for all $\mu,\nu\in\mathcal{P}_+$. Moreover, for a sequence $\{\mu_n\}_n$ in $\mathcal{P}_+$ and $\{k_n\}_n$ in $\mathbb{N}$ with $k_1<k_2<\cdots$ and $k_n\rightarrow\infty$ as $n\rightarrow\infty$, if there exists $\mu\in\mathcal{P}_+$ such that $\mu_n^{{\cup \hspace{-.52em}\lor} k_n} \xrightarrow{w} \mu$, then $\mu_n^{{\Box \hspace{-.55em} \lor}k_n}\xrightarrow{w} B_t^\lor(\mu)$ as $n\rightarrow\infty$. In particular, we have $B_1^\lor(B_{\text{II},\alpha})=F_{\text{II},\alpha}$ for all $\alpha>0$ (see \cite{U19}). In addition, we have the following relation.

\begin{lemma}\label{lem:B1lx}\cite{U19}
We have $B_1^\lor=\Lambda^\lor\circ\mathcal{X}^\lor$.
\end{lemma}

%%%%%%%%%%%%%%%%%%
%PROOF OF MAIN THEOREMS
%%%%%%%%%%%%%%%%%%

\section{Proof of main theorems}

%%%%%%%%%%
%THEOREM1.1
%%%%%%%%%%

\subsection{Proof of Theorem \ref{thm:free_additive_max}}

In this section, we firstly show that the operator $\Phi$ has a relation between free additive convolution and free max-convolution.

\begin{proof}[Proof of Theorem \ref{thm:free_additive_max}]
We may assume that $t>1$. If $\mu$ is a Dirac measure, then $\Phi(D_{1/t}(\mu^{\boxplus t}))=\mu=\Phi(\mu)^{\Box \hspace{-.55em} \lor t}$. Therefore we may assume that $\mu$ is not a Dirac measure. Note that the closure of interval $(\alpha_t,\omega)$ is the support of $\Phi(\mu)^{\Box \hspace{-.55em} \lor t}$, where 
\begin{align*}
\alpha_t:=\inf \left\{x: \Phi(\mu)([0,x])>1-\frac{1}{t}\right\}, \qquad \omega:=\sup\{x:\Phi(\mu)([0,x])<1\},
\end{align*}
for each $t>1$. Define $M_t:=\{y:\Phi(\mu)^{\Box \hspace{-.55em}\lor t}([0,y])\in(\mu^{\Box \hspace{-.55em}\lor t}(\{0\}),1)\}$. By Remark \ref{rem:support}, the closure of $M_t$ is also the support of $\Phi(\mu)^{\Box \hspace{-.55em} \lor t}$. Moreover, we have $M_t\subseteq (\alpha_t,\omega)$.

Next, we define
\begin{align*}
A_t:&=\{y: \Phi(D_{1/t}(\mu^{\boxplus t}))([0,y])\in(\mu^{\boxplus t}(\{0\}),1)\}\\
&=\left(\frac{a_{\mu^{\boxplus t}}}{t}, \frac{b_{\mu^{\boxplus t}}}{t}\right), \qquad t>1,
\end{align*}
where $a_{\mu^{\boxplus t}}$ and $b_{\mu^{\boxplus t}}$ were defined by \eqref{ab}, and the last equality holds by Lemma \ref{lem:pS}. By Remark \ref{rem:support}, the closure of $A_t$ is the support of $\Phi(D_{1/t}(\mu^{\boxplus t}))$ for each $t>1$.

We show that $A_t=M_t$ and 
\begin{align*}
\Phi(D_{1/t}(\mu^{\boxplus t}))([0,x])=t\Phi(\mu)([0,x])-(t-1),
\end{align*}
for all $x\in A_t=M_t$ and each $t>1$. For an arbitrary fixed $t>1$, we divided two cases to prove it.\\
\vspace{-3mm}\\
{\bf Case I}\hspace{2mm}$0\le \mu(\{0\}) \le 1-t^{-1}$: By Lemma \ref{lem:atom} and Lemma \ref{lem:free-max-atom}, we have 
\begin{align*}
\mu^{\boxplus t}(\{0\})=\mu^{\Box \hspace{-.55em} \lor t}(\{0\})=0. 
\end{align*}
For all $x\in M_t$, we get 
\begin{align*}
t\Phi(\mu)([0,x])-(t-1)\in (\mu^{\Box \hspace{-.55em} \lor t}(\{0\}),1 )=(0,1)=(\mu^{\boxplus t}(\{0\}),1).
\end{align*}
Since $M_t\subseteq (\alpha_t,\omega)\subseteq (a_\mu,b_\mu)$, we have 
\begin{align*}
x=\frac{1}{S_\mu(\Phi(\mu)([0,x])-1)},
\end{align*}
by Lemma \ref{lem:pS}. Therefore Lemma \ref{lem:Sfb} and Proposition \ref{prop:HMlimit} imply that
\begin{align*}
\Phi(\mu^{\boxplus t})([0,tx])&=\Phi(\mu^{\boxplus t}) \left( \left[0, \frac{t}{S_\mu(\Phi(\mu)([0,x])-1)}\right]\right)\\
&=\Phi(\mu^{\boxplus t})\left( \left[0, \frac{1}{S_{\mu^{\boxplus t}}\left(t(\Phi(\mu)([0,x])-1)\right)} \right]\right)\\
&=\Phi(\mu^{\boxplus t})\left( \left[0, \frac{1}{S_{\mu^{\boxplus t}}\left(\{t\Phi(\mu)([0,x])-(t-1)\}-1\right)} \right]\right)\\
&=t\Phi(\mu)([0,x])-(t-1).
\end{align*}
By Lemma \ref{lem:dP}, we have 
\begin{align*}
\Phi(D_{1/t}(\mu^{\boxplus t}))([0,x])&=D_{1/t}\circ\Phi(\mu^{\boxplus t})([0,x])\\
&=\Phi(\mu^{\boxplus t})([0,tx])\\
&=t\Phi(\mu)([0,x])-(t-1) \in (0,1),
\end{align*}
and therefore we have $x\in A_t$. Hence we have $M_t\subseteq A_t$.

Next we show that for all $x\in (a_{\mu^{\boxplus t}}, b_{\mu^{\boxplus t}})$, we have $x/t\in (\alpha_t,\omega)$. For any $x\in (a_{\mu^{\boxplus t}}, b_{\mu^{\boxplus t}})$, we have
\begin{align*}
\Phi(\mu^{\boxplus t})([0,x])=S^{-1}_{\mu^{\boxplus t}}\left(\frac{1}{x}\right)+1
\end{align*}
by Lemma \ref{lem:pS}. Since $S_{\mu^{\boxplus t}}(S_{\mu^{\boxplus t}}^{-1}(1/x))=1/x$, we have
\begin{align*}
S_\mu\left(\frac{1}{t}S^{-1}_{\mu^{\boxplus t}}\left(\frac{1}{x}\right)\right)=\frac{t}{x},
\end{align*}
by Lemma \ref{lem:Sfb}. Therefore we obtain
\begin{align*}
S_{\mu^{\boxplus t}}^{-1}\left(\frac{1}{x}\right)=tS_\mu^{-1}\left(\frac{t}{x}\right),
\end{align*}
by Lemma \ref{lem:image}. Thus, we get $\Phi(\mu^{\boxplus t})([0,x])=tS_{\mu}^{-1}(t/x)+1$, and therefore
\begin{align*}
x=\frac{t}{S_\mu\left(\frac{1}{t}\Phi(\mu^{\boxplus t})([0,x])-\frac{1}{t}+1-1 \right)}.
\end{align*}
Since
\begin{align*}
\frac{1}{t}\Phi(\mu^{\boxplus t})([0,x])-\frac{1}{t}+1\in \left(1-\frac{1}{t},1 \right)\subseteq (\mu(\{0\}),1),
\end{align*}
we have 
\begin{align*}
\Phi(\mu)([0,x/t])&=\Phi(\mu)\left( \left[0,\frac{1}{ S_\mu\left(\frac{1}{t}\Phi(\mu^{\boxplus t})([0,x])-\frac{1}{t}+1-1 \right)}\right]\right)\\
&=\frac{1}{t}\Phi(\mu^{\boxplus t})([0,x])-\frac{1}{t}+1\in \left( 1-\frac{1}{t},1\right),
\end{align*}
by Proposition \ref{prop:HMlimit}. This implies that $x/t\in (\alpha_t,\omega)$. Therefore we have 
\begin{align*}
A_t=\left( \frac{a_{\mu^{\boxplus t}}}{t},\frac{b_{\mu^{\boxplus t}}}{t}\right)\subseteq (\alpha_t,\omega)=M_t.
\end{align*}
Finally, we get the following properties:
\begin{align*}
\Phi(D_{1/t}(\mu^{\boxplus t}))([0,x])&=t\Phi(\mu)([0,x])-(t-1) 
\end{align*}
for all $x\in A_t=M_t$. Moreover, the support of $\Phi(D_{1/t}(\mu^{\boxplus t}))$ is equal to the support of $\Phi(\mu)^{\Box \hspace{-.55em} \lor t}$ since $A_t=M_t$. Therefore $\Phi(D_{1/t}(\mu^{\boxplus t}))=\Phi(\mu)^{\Box \hspace{-.55em} \lor t}$.\\
\vspace{-2mm}\\
{\bf Case II} $\mu(\{0\})>1-t^{-1}$:  By Lemma \ref{lem:atom}, we have 
\begin{align*}
\mu^{\boxplus t}(\{0\})=t\mu(\{0\})-(t-1)>0,
\end{align*}
and therefore $a_{\mu^{\boxplus t}}=0$. By Lemma \ref{lem:free-max-atom}, we also have
\begin{align*}
\mu^{\Box \hspace{-.55em} \lor t}(\{0\})=t\mu(\{0\})-(t-1)>0,
\end{align*}
and therefore $\alpha_t=0$. For any $x\in M_t=(0,\omega)$, we get 
\begin{align*}
t\Phi(\mu)([0,x])-(t-1)\in (\mu^{\Box \hspace{-.55em} \lor t}(\{0\}),1)=(\mu^{\boxplus t}(\{0\}),1).
\end{align*}
Since $M_t=(0,\omega)\subseteq (0,b_\mu)=(a_\mu,b_\mu)$, we get
\begin{align*}
\Phi(D_{1/t}(\mu^{\boxplus t}))([0,x])=t\Phi(\mu)([0,x])-(t-1)\in(\mu^{\boxplus t}(\{0\}),1),
\end{align*} 
by the same way in Case I. Hence we have $x\in A_t$. Therefore $M_t\subseteq A_t$.

Next for $x\in(a_{\mu^{\boxplus t}},b_{\mu^{\boxplus t}})=(0,b_{\mu^{\boxplus t}})$, we show that $x/t\in (\alpha_t,\omega)=(0,\omega)$. Since 
\begin{align*}
\frac{1}{t}\Phi(\mu^{\boxplus t})([0,x])-\frac{1}{t}+1 &\in \left( \frac{1}{t}\mu^{\boxplus t}(\{0\})-\frac{1}{t}+1,1\right)\\
&=\left(\frac{1}{t}(t\mu(\{0\})-(t-1))-\frac{1}{t}+1,1 \right)\\
&=(\mu(\{0\}),1),
\end{align*}
we have 
\begin{align*}
\Phi(\mu)\left( \left[0,\frac{x}{t} \right]\right)=\frac{1}{t}\Phi(\mu^{\boxplus t})([0,x])-\frac{1}{t}+1 &\in (\mu(\{0\}),1)\subseteq \left( 1-\frac{1}{t},1\right).
\end{align*}
by the same way in Case I, where the last inclusion holds since $\mu(\{0\})>1-t^{-1}$. Therefore $x/t\in (0,\omega)$, and hence 
\begin{align*}
A_t=\left(0,\frac{b_{\mu^{\boxplus t}}}{t}\right)\subseteq (0,\omega)=M_t.
\end{align*}
Finally, we get the following properties:
\begin{align*}
\Phi(D_{1/t}(\mu^{\boxplus t}))([0,x])&=t\Phi(\mu)([0,x])-(t-1) 
\end{align*}
for all $x\in A_t=M_t$. Moreover, the support of $\Phi(D_{1/t}(\mu^{\boxplus t}))$ is equal to the support of $\Phi(\mu)^{\Box \hspace{-.55em} \lor t}$ since $A_t=M_t$. Therefore $\Phi(D_{1/t}(\mu^{\boxplus t}))=\Phi(\mu)^{\Box \hspace{-.55em} \lor t}$.
\end{proof}

%%%%%%%%%%
%THEOREM1.2
%%%%%%%%%%

\subsection{Proof of Theorem \ref{thm:Boolean_additive_max}}
In this section, we show Theorem \ref{thm:Boolean_additive_max} as follows.
\begin{proof}[Proof of Theorem \ref{thm:Boolean_additive_max}]  
If $\mu$ is a Dirac measure, then $\Phi(D_{1/t}(\mu^{\uplus t}))=\mu=\Phi(\mu)^{\cup \hspace{-.52em} \lor t}$.  Therefore we may assume that $\mu$ is not a Dirac measure. By Corollary \ref{cor:Boolean}, the eqaution \eqref{eq:Boolean_add_atom} and Lemma \ref{lem:Boolean-max-atom}, if $\mu(\{0\})=0$ then $\mu^{\uplus t}(\{0\})=\mu^{\cup \hspace{-.52em} \lor t}(\{0\})=0$ and if $\mu(\{0\})\neq 0$, then
\begin{align*}
\mu^{\uplus t}(\{0\})=\mu^{\cup \hspace{-.52em} \lor t}(\{0\})=\frac{\mu(\{0\})}{t-(t-1)\mu(\{0\})}.
\end{align*}  
Define $M_t:=\{x:\Phi(\mu)^{\cup \hspace{-.52em}\lor t}([0,x])\in(\mu^{\cup \hspace{-.52em} \lor t}(\{0\}),1)\}$. For all $y\in M_t$, we have
\begin{align*}
\Phi(\mu)^{\cup \hspace{-.52em}\lor t}([0,y])\in (\mu^{\cup \hspace{-.52em} \lor t}(\{0\}),1)=\left( \frac{\mu(\{0\})}{t-(t-1)\mu(\{0\})},1\right).
\end{align*}
Then 
\begin{align*}
\Phi(\mu)([0,y])=\frac{t\Phi(\mu)^{\cup \hspace{-.52em}\lor t}([0,y])}{1+(t-1)\Phi(\mu)^{\cup \hspace{-.52em}\lor t}([0,y])}&\in\left(\frac{\frac{t\mu(\{0\})}{t-(t-1)\mu(\{0\})}}{1+(t-1)\frac{\mu(\{0\})}{t-(t-1)\mu(\{0\})}},1\right)\\
&=(\mu(\{0\}),1)
\end{align*}
Hence $y\in \{x:\Phi(\mu)([0,x])\in(\mu(\{0\}),1)\}=(a_\mu,b_\mu)$, and therefore $M_t\subseteq (a_\mu,b_\mu)$. It is clear that $(a_\mu,b_\mu)\subseteq M_t$. Finally, we have 
\begin{align*}
M_t=\{x:\Phi(\mu)([0,x])\in(\mu(\{0\}),1)\}=(a_\mu,b_\mu).
\end{align*}

The set $M_t$ does not depend on $t$, and therefore we denote by $M$ as the set $M_t$. By Remark \ref{rem:support}, the closure of $M$ is the support of $\Phi(\mu)^{\cup \hspace{-.52em}\lor t}$ for all $t>0$. Note that the support of $\Phi(\mu)^{\cup \hspace{-.52em}\lor t}$ coincides with the support of $\Phi(\mu)$. 

Next we define $A_t:=\{x: \Phi(D_{1/t}(\mu^{\uplus t}))([0,x])\in (\mu^{\uplus t}(\{0\}),1)\}$. Then we have
\begin{align*}
A_t=\left( \frac{a_{\mu^{\uplus t}}}{t},\frac{b_{\mu^{\uplus t}}}{t}\right).
\end{align*}
By Remark \ref{rem:support}, the closure of $A_t$ is the support of $\Phi(D_{1/t}(\mu^{\uplus t}))$ for each $t>0$. To get this theorem, we show that $A_t=M$ and 
\begin{align*}
\Phi(D_{1/t}(\mu^{\uplus t}))([0,x])=\frac{\Phi(\mu)([0,x])}{t-(t-1)\Phi(\mu)([0,x])}
\end{align*}
for all $x\in A_t=M$ and each $t>0$. Give an arbitrary fixed $t>0$.

For all $x\in M$, we have
\begin{align*}
\frac{\Phi(\mu)([0,x])}{t-(t-1)\Phi(\mu)([0,x])}\in (\mu^{\cup \hspace{-.52em} \lor t}(\{0\}),1)=(\mu^{\uplus t}(\{0\}),1).
\end{align*}
Since $M=(a_\mu,b_\mu)$, by Lemma \ref{lem:pS}, we have
\begin{align*}
x=\frac{1}{S_\mu(\Phi(\mu)([0,x])-1)}.
\end{align*}
Applying Lemmas \ref{lem:Sfb} and Proposition \ref{prop:HMlimit}, we have
\begin{align*}
\Phi(\mu^{\uplus t})([0,tx])&=\Phi(\mu^{\uplus t})\left(\left[0, \frac{t}{S_\mu(\Phi(\mu)([0,x])-1)} \right] \right)\\
&=\Phi(\mu^{\uplus t})\left( \left[0, \frac{1}{S_{\mu^{\uplus t}}(\frac{\Phi(\mu)([0,x])}{t-(t-1)\Phi(\mu)([0,x])}-1) }\right]\right)\\
&=\frac{\Phi(\mu)([0,x])}{t-(t-1)\Phi(\mu)([0,x])}\\
&=\Phi(\mu)^{\cup \hspace{-.52em} \lor t}([0,x]).
\end{align*}
By Lemma \ref{lem:dP}, we have 
\begin{align*}
\Phi(D_{1/t}(\mu^{\uplus t}))([0,x])&=D_{1/t}\circ\Phi(\mu^{\uplus t})([0,x])\\
&=\Phi(\mu^{\uplus t})([0,tx])\\
&=\Phi(\mu)^{\cup \hspace{-.52em} \lor t}([0,x]).
\end{align*}

Next we show that $A_t=M$. By Lemma \ref{lem:image}, we have
\begin{align*}
\left(\frac{1}{b_{\mu^{\uplus t}}},\frac{1}{a_{\mu^{\uplus t}}}\right)&=S_{\mu^{\uplus t}}\left( (\mu^{\uplus t}(\{0\})-1,0)\right)\\
&=S_{\mu^{\uplus t}}\left(\left( \frac{t(\mu(\{0\})-1)}{t-(t-1)\mu(\{0\})},1\right) \right)=\frac{1}{t}\left( \frac{1}{b_\mu},\frac{1}{a_\mu}\right),
\end{align*}
where the last equation holds by Lemma \ref{lem:Sfb}. Therefore $M=(a_\mu,b_\mu)=A_t$. Thus the support of $\Phi(D_{1/t}(\mu^{\uplus t}))$ is equal to the support of $\Phi(\mu)^{\cup \hspace{-.52em} \lor t}$. 
\end{proof}

\begin{remark}
Consider $\mu\in\mathcal{P}_+\setminus\{\delta_0\}$. Since $A_t=(a_\mu,b_\mu)$ in the proof of Theorem \ref{thm:Boolean_additive_max}, we get 
\begin{align*}
t \int_0^\infty x^{-1} d\mu^{\uplus t}(x)&=\int_0^\infty x^{-1}d\mu(x), \qquad t>0.
\end{align*}
\end{remark}

%%%%%%%%%%
%THEOREM1.3
%%%%%%%%%%

\subsection{Proof of Theorem \ref{thm:intertwiner}}
By using Theorems \ref{thm:free_additive_max} and \ref{thm:Boolean_additive_max}, we know that Belinschi-Nica semigroup is closely intertwined with max-Belinschi-Nica semigroup via the operator $\Phi$.

\begin{proof}[Proof of Theorem \ref{thm:intertwiner}]
If $\mu$ is a Dirac measure, then $\Phi(B_t(\mu))=\mu=B_t^\lor(\Phi(\mu))$. Therefore we may assume that $\mu$ is not a Dirac measure. By Theorems \ref{thm:free_additive_max}, \ref{thm:Boolean_additive_max} and Lemma \ref{lem:dP}, we have
\begin{align*}
\Phi(B_t(\mu))&=D_{\frac{1}{1+t}}\left(\Phi(\mu^{\boxplus (1+t)})^{\cup \hspace{-.52em} \lor \frac{1}{1+t}}\right)\\
&=D_{\frac{1}{1+t}} \left( \left(D_{1+t} \left(\Phi(\mu)^{\Box \hspace{-.55em} \lor (1+t)} \right)  \right)^{\cup \hspace{-.52em} \lor \frac{1}{1+t}}\right)\\
&=D_{\frac{1}{1+t}}\circ D_{1+t} \left(\left( \Phi(\mu)^{\Box \hspace{-.55em} \lor (1+t)}\right)^{\cup \hspace{-.52em} \lor \frac{1}{1+t}} \right) =B_t^\lor(\Phi(\mu)).
\end{align*}
\end{proof}

%%%%%%%%%%
%THEOREM1.4
%%%%%%%%%%

\subsection{Proof of Theorem \ref{classical_additive_max2}}
In this section, we prove Theorem \ref{classical_additive_max2}. Before a proof of this theorem, we prepare the following important maps. Let $\text{FID}_+$ be the set of all freely infinitely divisible distributions on $[0,\infty)$.  Recall that $B_t(\mu):=(\mu^{\boxplus (1+t)}))^{\uplus \frac{1}{1+t}}\in\text{FID}_+$ for any $\mu\in \mathcal{P}_+$ and $t\ge 1$ (see \cite{BN08}). Hence we can define the operator $\mathcal{X}:\mathcal{P}_+\rightarrow\text{ID}_+$ is defined by
\begin{align*}
\mathcal{X}:=\Lambda^{-1}\circ B_1,
\end{align*}
where $\Lambda$ is called the {\it Bercovici-Pata bijection} and $\Lambda^{-1}(\text{FID}_+)=\text{ID}_+$ (for details, see \cite{BP99, BT02}).

For any $\nu\in\text{FID}_+$ and $t \ge 0$, the measure $\nu^{\uplus (1+t)}$ is also in $\text{FID}_+$ and we obtain $B_t^{-1}(\nu)=(\nu^{\uplus (1+t)})^{\boxplus \frac{1}{1+t}}$. Hence $\mathcal{X}^{-1}=B_1^{-1}\circ \Lambda$, and therefore $\mathcal{X}$ is a bijection from $\mathcal{P}_+$ to $\text{ID}_+$. It is called the {\it Boolean-classical Bercovici-Pata bijection} (see \cite{BP99,BN08}). Note that $\mathcal{X}^{-1}\circ D_c=D_c\circ\mathcal{X}^{-1}$ and $\mathcal{X}^{-1}(\mu^{\ast t})=\mathcal{X}^{-1}(\mu)^{\uplus t}$ for all $\mu\in\text{ID}_+$ and $c,t>0$. 

%Next, we define the map $\mathcal{X}^\lor:\mathcal{P}_+\rightarrow\mathcal{P}_+$ by setting
%\begin{align*}
%\mathcal{X}^{\lor}(\mu)([0,\cdot]):=\exp\left[ 1-\frac{1}{\mu([0,\cdot])}\right], \qquad \mu\in\mathcal{P}_+,
%\end{align*}
%and if $\mu([0,x])=0$, then we understand $\mathcal{X}^\lor(\mu)([0,x])=0$ for such $x\ge 0$. The map $\mathcal{X}^{\lor}$ is called the {\it Boolean-classical max-Bercovici-Pata bijection} (see \cite{VV18, U19}). 

Finally, we define the operator $\Psi:\text{ID}_+\rightarrow\mathcal{P}_+$ by setting
\begin{align*}
\Psi:=\mathcal{X}^\lor \circ \Phi\circ \mathcal{X}^{-1},
\end{align*}
where $\mathcal{X}^\lor$ was defined in Section 2.3.3.

\begin{proof}[Proof of Theorem \ref{classical_additive_max2}]
Consider $\mu\in\mathcal{P}_+$ and $t>0$. By Theorem \ref{thm:Boolean_additive_max}, we have
\begin{align*}
\Psi(D_{1/t}(\mu^{\ast t}))([0,\cdot])&=\mathcal{X}^{\lor}\circ \Phi\circ\mathcal{X}^{-1}(D_{1/t}(\mu^{\ast t}))([0,\cdot])\\
&=\mathcal{X}^{\lor}\circ\Phi (D_{1/t}(\mathcal{X}^{-1}(\mu)^{\uplus t}))([0,\cdot])\\
&=\mathcal{X}^{\lor}(\Phi(\mathcal{X}^{-1}(\mu))^{\cup \hspace{-.52em} \lor t})([0,\cdot])\\
&=\exp\left[ 1-\frac{1}{\Phi(\mathcal{X}^{-1}(\mu))^{\cup \hspace{-.52em} \lor t}([0,\cdot])}\right]\\
&=\exp\left[ t\left(1- \frac{1}{\Phi(\mathcal{X}^{-1}(\mu))([0,\cdot])}\right)\right].
\end{align*}
On the other hand, we have
\begin{align*}
\Psi(\mu)^{\lor t}([0,\cdot])&=\left(\mathcal{X}^{\lor}\circ \Phi\circ\mathcal{X}^{-1}(\mu)([0,\cdot])\right)^t\\
&=\exp\left[ t\left(1- \frac{1}{\Phi(\mathcal{X}^{-1}(\mu))([0,\cdot])}\right)\right].
\end{align*}
Therefore we have $\Psi(D_{1/t}(\mu^{\ast t}))=\Psi(\mu)^{\lor t}$.
\end{proof}

In addition, we conclude that the operator $\Phi$ is intertwined with the operator $\Psi$ as follows.

\begin{proposition}
We have $\Lambda^{\lor}\circ\Psi=\Phi\circ\Lambda$.
\end{proposition}
\begin{proof}
By Lemma \ref{lem:B1lx} and Theorem \ref{thm:intertwiner}, we have
\begin{align*}
\Lambda^{\lor}\circ \Psi&=(\Lambda^\lor\circ\mathcal{X}^\lor )\circ \Phi\circ \mathcal{X}^{-1}\\
&=(B_1^\lor \circ \Phi\circ B_1^{-1})\circ\Lambda\\
&=\Phi\circ\Lambda.
\end{align*}
\end{proof}

According to discussions at Section 3, we get the commutative diagram.

\[
   \begin{CD}
\text{(Classical)} @. \text{(Free)} @. \text{(Boolean)}\\
\text{ID}_+  @>{\Lambda}>>  \text{FID}_+  @<{B_1}<<  \text{BID}_+ @. \hspace{3mm}\text{(Additive)}\\
@V{\Psi}VV  @VV{\Phi}V  @VV{\Phi}V\\
\mathcal{P}_+ @>>{\Lambda^\lor}>  \mathcal{P}_+ @<<{B_1^\lor}< \mathcal{P}_+ @. \hspace{3mm}\text{(Max)}\\
  \end{CD}
\]
The class $\text{BID}_+$ is the set of all Boolean infinitely divisible distributions on $[0,\infty)$. It is well known that $\text{BID}_+=\mathcal{P}_+$ (see \cite{SW97}).

%%%%%%%%%%%%%%%%%%%%%%%%
% Examples
%%%%%%%%%%%%%%%%%%%%%%%%

\section{Examples}

In this section, we give several examples of probability measures in the classes $\Phi(\mathcal{P}_+)$ and $\Psi(\mathcal{P}_+)$
%%%%%%%%%%%%%%%%%%%%%%%%
% STABLE LAWS AND EXTREME VALUES
%%%%%%%%%%%%%%%%%%%%%%%%

\subsection{Stable laws and extreme values}

In this section, we give relations between (strictly) stable laws and extreme value distributions. Let $\mathcal{A}$ be the set of {\it admissible parameters}:
\begin{align*}
\mathcal{A}:=\{(\alpha,\rho):\alpha\in(0,1],\rho\in[0,1]\}\cup \{(\alpha,\rho):\alpha\in(1,2],\rho\in[1-\alpha^{-1},\alpha^{-1}]\}.
\end{align*}
Consider $(\alpha,\rho)\in\mathcal{A}$. Denote by $c_{\alpha,\rho}$ the {\it classical strictly stable law} (see e.g. \cite{Sato}), $f_{\alpha,\rho}$ the {\it free strictly stable law} (see \cite{BV93,BP99}) and $b_{\alpha,\rho}$ the {\it Boolean strictly stable law} (see \cite{SW97}). In particular, we define $c_{\alpha}^+:=c_{\alpha,1}\in \text{ID}_+$, $f_{\alpha}^+:=f_{\alpha,1}\in\text{FID}_+$ and $b_{\alpha}^+:=b_{\alpha,1}\in\mathcal{P}_+$. Note that $c_1^+=f_1^+=b_1^+=\delta_1$. For all $\alpha\in(0,1)$, the strictly stable laws $c_{\alpha}^+$, $f_{\alpha}^+$ and $b_{\alpha}^+$ are not delta measure. Thus we may assume that $\alpha\in (0,1)$ in this section.

In \cite{AH16}, we know a relation between Boolean stable law and Boolean extreme value distribution via the operator $\Phi$.
\begin{example}\label{lem:Booleanstable}
Consider $\alpha\in(0,1)$. Then $\Phi(b_{\alpha}^+)([0,\cdot])=B_{\text{II},\frac{\alpha}{1-\alpha}}(\cdot)$. 
\end{example}

Next, we give a relation between free stable law and free extreme value distribution via the operator $\Phi$. 
\begin{example}\label{cor:freestable}
Consider $\alpha\in(0,1)$. Then $\Phi(f_{\alpha}^+)([0,\cdot])=F_{\text{II},\frac{\alpha}{1-\alpha}}(\cdot)$.
\end{example}
\begin{proof}
By \cite{BN08,BP99}, we have $B_1(b_\alpha^+)=f_{\alpha}^+$. Moreover, by \cite{U19}, we have  $B_1^\lor(B_{\text{II},\frac{\alpha}{1-\alpha}})=F_{\text{II},\frac{\alpha}{1-\alpha}}$. By Theorem \ref{thm:intertwiner} and Example \ref{lem:Booleanstable}, we have
 \begin{align*}
 \Phi(f_\alpha^+)([0,\cdot])&=\Phi(B_1(b_\alpha^+))([0,\cdot])\\
 &=B_1^\lor(\Phi(b_{\alpha}^+))([0,\cdot]))\\
 &=B_1^\lor(B_{\text{II},\frac{\alpha}{1-\alpha}}(\cdot))=F_{\text{II},\frac{\alpha}{1-\alpha}}(\cdot).
 \end{align*}
 Therefore $\Phi(f_{\alpha}^+)([0,\cdot])=F_{\text{II},\frac{\alpha}{1-\alpha}}(\cdot)$.
\end{proof}

Finally, we obtain a relation between classical stable law and classical extreme value distribution via the operator $\Psi$.
\begin{example}
Consider $\alpha\in(0,1)$. Then $\Psi(c_\alpha^+)([0,\cdot])=C_{\text{II},\frac{\alpha}{1-\alpha}}(\cdot)$.
\end{example}
\begin{proof}
By \cite{BN08,BP99}, we have $\mathcal{X}^{-1}(c_\alpha^+)=b_\alpha^+$. Moreover, by \cite{VV18}, we have $\mathcal{X}^{\lor}(B_{\text{II},\frac{\alpha}{1-\alpha}})=C_{\text{II},\frac{\alpha}{1-\alpha}}$. By Example \ref{lem:Booleanstable}, we have
\begin{align*}
\Psi(c_\alpha^+)([0,\cdot])&=\mathcal{X}^\lor \circ \Phi(b_\alpha^+)([0,\cdot])\\
&=\mathcal{X}^\lor(B_{\text{II},\frac{\alpha}{1-\alpha}}(\cdot))\\
&=C_{\text{II},\frac{\alpha}{1-\alpha}}(\cdot).
\end{align*}
Therefore $\Psi(c_\alpha^+)([0,\cdot])=C_{\text{II},\frac{\alpha}{1-\alpha}}(\cdot)$. 
\end{proof}

%%%%%%%%%%%%%%%%%%%%%%%%

\subsection{Marchenko-Pastur law and uniform distribution}

In this section, we mention that $\Phi$ connects the Marchenko-Pastur law to the uniform distribution on $(0,1)$. Denote by $\pi$ the {\it Marchenko-Pastur law} (or {\it free Poisson law}), that is,
\begin{align*}
\pi(dx) :=\frac{1}{2\pi}\sqrt{\frac{4-x}{x}}\mathbf{1}_{(0,4)}(x)dx.
\end{align*}
This distribution appears as the limit of eigenvalue distributions of Wishart matrices as the size of the random matrices goes to infinity. 

It is known that the S-transform of the Marchenko-Pastur law is given by
\begin{align*}
S_\pi(z)=\frac{1}{z+1}, \qquad z\in(-1,0).
\end{align*}

For a reader's convenience, we give a proof of the above result. Denote by $_2F_1(a,b,c;z)$ the {\it Gauss hypergeometric series}:
\begin{align*}
_2F_1(a,b,c;z):=\sum_{n=0}^\infty \frac{(a)_n(b)_n}{(c)_n}\frac{z^n}{n!},
\end{align*}
where $c\neq0,-1,-2,\cdots$ and $(a)_n$ denotes the {\it Pochhammer symbol}: $(a)_n:=\prod_{k=0}^{n-1}(a+k)$, $(a)_0=1$. If $\text{Re}(c)>\text{Re}(b)>0$ and $|z|<1$, then we obtain
\begin{align*}
_2F_1(a,b,c;z)=\frac{1}{B(b,c-b)}\int_0^1 x^{b-1}(1-x)^{c-b-1}(1-zx)^{-a}dx,
\end{align*}
where $B(t,s)$ is the Beta function. In particular, we get
\begin{align*}
_2F_1(1,s,3;z)=-\frac{2\left( (s-2)z+1-(1-z)^{2-s}\right)}{(s-2)(s-1)z^2}
\end{align*}
for $s>0$, $s\neq1,2,$ (see e.g. \cite[Lemma 3.5]{MSU19}). To obtain the S-transform $S_\pi$, we firstly compute $\Psi_\pi^{-1}$. For $z \in\mathbb{C}$ with $|z|<1/4$, we have
\begin{align*}
\Psi_\pi(z)&=\int_0^4 \frac{zx}{1-zx}\cdot \frac{1}{2\pi}\sqrt{\frac{4-x}{x}}dx\\
&=\frac{8z}{\pi}\int_0^1 x^{1/2}(1-x)^{1/2}(1-4zx)^{-1}dx\\
&=\frac{8z}{\pi} \times B\left(\frac{3}{2},\frac{3}{2}\right)  {_2F_1}\left(1,\frac{3}{2},3;4z\right)\\
&=\frac{8z}{\pi}\times \frac{\pi}{8} \times\left( \frac{-2z+1-(1-4z)^{1/2}}{2z^2}\right)\\
&=\frac{-2z+1-(1-4z)^{1/2}}{2z}.
\end{align*}
Moreover, this is analytic on a neighborhood of $(-\infty,0)$. By the identity theorem, we get
\begin{align*}
\Psi_\pi(z)=\frac{-2z+1-(1-4z)^{1/2}}{2z}
\end{align*}
for $z$ in a neighborhood of $(-\infty,0)$. Then we get
\begin{align*}
z=\frac{-2\Psi_\pi^{-1}(z)+1-\sqrt{1-4\Psi_\pi^{-1}(z)}}{2\Psi_\pi^{-1}(z)}, \qquad z\in (-1,0)
\end{align*}
and therefore 
\begin{align*}
\Psi_\pi^{-1}(z)=\frac{z}{(z+1)^2}, \qquad z\in (-1,0).
\end{align*}
Hence $S_\pi(z)=\frac{z+1}{z}\Psi_\pi^{-1}(z)=\frac{1}{z+1}$ for all $z\in(-1,0)$. Therefore we obtain the following example.

\begin{example}\label{ex:marchenko-uniform}
We have
\begin{align*}
\Phi(\pi)\left([0,x]\right)=x, \qquad x\in (0,1).
\end{align*}
Thus $\Phi(\pi)=U(0,1)$, where $U(0,1)$ is the uniform distribution on $(0,1)$.
\end{example}

%%%%%%%%%%%%%%%%%%%%%%%%
% POISSON DISTRIBUTIONS
%%%%%%%%%%%%%%%%%%%%%%%%

\subsection{Poisson law and max-compound Poisson law}

We find a relation between the Poisson law and the max-compound Poisson law with the uniform distribution on (0,1) via the map $\Psi$.  Denote by $Po(\lambda)$ the {\it Poisson law with parameter $\lambda>0$}, that is, 
\begin{align*}
Po(\lambda)=\sum_{k=0}^\infty \frac{\lambda^ke^{-\lambda}}{k!}\delta_k.
\end{align*}
Note that $\Lambda(Po(1))=\pi$ (see \cite{BP99}) and $\mathcal{X}^{-1}(Po(1))=\frac{1}{2}\delta_0+\frac{1}{2}\delta_2$ (see \cite{SW97}). Hence we obtain the following relation.

\begin{example}\label{max-Poisson}
We have 
\begin{align*}
\Psi(Po(1))=\Pi^{\lor}(U(0,1))=\frac{1}{e}\delta_0+e^{-1+x}\mathbf{1}_{(0,1)}(x)dx.
\end{align*}
\end{example}
\begin{proof}
We simply calculate the left hand side as follows. For all $x\in (0,1)$,
\begin{align*}
\Psi(Po(1))([0,x])&=(\mathcal{X}^\lor \circ \Phi )(\mathcal{X}^{-1}(Po(1)))([0,x])\\
&=\mathcal{X}^{\lor}\left(\Phi\left(\frac{1}{2}\delta_0+\frac{1}{2}\delta_2\right)\right)([0,x])\\
&=\exp\left[1-\frac{1}{\Phi\left(\frac{1}{2}\delta_0+\frac{1}{2}\delta_2\right)([0,x])}\right].
\end{align*}
Put $\sigma:=\frac{1}{2}\delta_0+\frac{1}{2}\delta_2$. Then we get $\Psi_\sigma(z)=\frac{z}{1-2z}$. Then we get $\Psi_\sigma^{-1}(z)=\frac{z}{1+2z}$ for all $z\in(-\frac{1}{2},0)$. Hence we have
\begin{align*}
S_\sigma(z)=\frac{1+z}{1+2z}, \qquad z\in \left( -\frac{1}{2},0\right).
\end{align*}
Thus we get
\begin{align*}
\Phi(\sigma)\left( \left[0, \frac{2x-1}{x}\right]\right)=x, \qquad x\in\left(\frac{1}{2},1\right),
\end{align*}
and therefore 
\begin{align}\label{eq:uniform}
\Phi(\sigma)([0,x])=\frac{1}{2-x}, \qquad x\in (0,1).
\end{align}
Hence we obtain
\begin{align*}
\Psi(Po(1))([0,x])=\exp(-1+x), \qquad x\in(0,1).
\end{align*}
Therefore its density function is given by
\begin{align*}
\frac{d\Psi(Po(1))}{dx}(x)=e^{-1+x}\mathbf{1}_{(0,1)}(x).
\end{align*}
Moreover, $\Psi(Po(1))(\{0\})=\mathcal{X}^{\lor}(\sigma)(\{0\})=e^{-1}$. 

Note that the probability measure in right hand side of Example \ref{max-Poisson} is the max-compound Poisson law with the uniform distribution on $(0,1)$. 
\end{proof}

By Theorem \ref{classical_additive_max2}, for $\lambda>0$, we get
\begin{align*}
\Psi(Po(\lambda))([0,x])&=\Psi(Po(1)^{\ast \lambda})([0,x])\\
&= \Psi(Po(1))^{\lor \lambda}([0,\lambda^{-1}x])\\
&=\Psi(Po(1))([0,\lambda^{-1}x])^\lambda\\
&=\exp(\lambda(-1+\lambda^{-1}x))=\exp(-\lambda+x),
\end{align*}
for all $x\in(0,\lambda)$. Thus, for all $\lambda>0$, we have
\begin{align*}
\Psi(Po(\lambda))=e^{-\lambda}\delta_0+e^{-\lambda+x}\mathbf{1}_{(0,\lambda)}(x)dx.
\end{align*}

%%%%%%%%%%%%%%%%%%
%FREE REGULAR DISTRIBUTIONS
%%%%%%%%%%%%%%%%%%

\subsection{Free regular distributions}

By Example \ref{ex:marchenko-uniform} and \eqref{eq:uniform}, we get
\begin{align*}
\Phi(\pi)([0,\cdot])=\max\left\{ 0, 2-\frac{1}{\Phi(\frac{1}{2}\delta_0+\frac{1}{2}\delta_2)([0,\cdot])}\right\}.
\end{align*}

More generally, we get the above formula in case of free regular distributions. A probability measure $\mu$ on $[0,\infty)$ is said to be {\it free regular} if $\mu$ is freely infinitely divisible and $\mu^{\boxplus t}\in\mathcal{P}_+$ for all $t>0$. For example, the Marchenko-Pastur law $\pi$ is free regular. Recently, we proved that the Fuss-Catalan distribution $\mu(p,r)$ is free regular for $1\le r \le \min\{p-1,p/2\}$ or $p=r=1$ (see \cite{MSU19}). A free regular distribution was introduced by \cite{S11}. It is known that this distribution is a marginal law of free subordinator (see \cite{AHS13}).  The class of free regular distributions does not coincide with $\text{FID}_+$ (see also \cite{S11}). By Theorem \ref{thm:intertwiner}, we can get the following formula.

\begin{proposition}\label{cor:freeregular}
For any free regular distribution $\mu$, there exists a unique $\sigma\in\mathcal{P}_+$ such that
\begin{align*}
\Phi(\mu)([0,\cdot])=\max\left\{0, 2-\frac{1}{\Phi(\sigma)([0,\cdot])}\right\}.
\end{align*}
\end{proposition}
\begin{proof}
If $\mu$ is a Dirac measure, then we can take $\sigma=\mu$ to satisfy the above equation. Therefore we may assume that $\mu$ is not a Dirac measure. In \cite{AHS13}, there exists a unique $\sigma\in\mathcal{P}_+$ not being a Dirac measure such that $\mu=B_1(\sigma)$. By Theorem \ref{thm:intertwiner},
\begin{align*}
\max\left\{0, 2-\frac{1}{\Phi(\sigma)([0,\cdot])}\right\}&=B_1^\lor(\Phi(\sigma))([0,\cdot])\\
&=\Phi(B_1(\sigma))([0,\cdot])=\Phi(\mu)([0,\cdot]).
\end{align*}
Therefore we obtain the above equation.
\end{proof}

\subsection{Infinitely divisible distributions with regular L\'{e}vy-Khintchine representations}

In the proof of Example \ref{max-Poisson}, we have
 \begin{align*}
\Psi(Po(1))([0,x])=\exp\left[1-\frac{1}{\Phi\left(\frac{1}{2}\delta_0+\frac{1}{2}\delta_2\right)([0,x])}\right].
\end{align*}
More generally, we get the above formula in case of infinitely divisible distributions with regular L\'{e}vy-Khintchine representations which are marginal distributions of subordinators (see e.g. \cite{Ber99, Sato}). A probability measure $\mu\in\text{ID}_+$ is said to have a {\it regular L\'{e}vy-Khintchine representation} (in this paper, for short, we say that $\mu$ is {\it regular}) if its characteristic function is written by
\begin{align*}
\int_\mathbb{R} e^{izx} \mu(dx)=\exp\left(i\eta z+\int_0^\infty(e^{izx}-1)\nu(dx)\right),\qquad z\in\mathbb{R},
\end{align*}
where $\eta\ge 0$ and $\nu$ is the L\'{e}vy measure satisfying $\int_0^\infty (1\land x)\nu(dx)<\infty$ and $\nu((-\infty,0])=0$. It is known that $\mu\in\text{ID}_+$ is regular if and only if $\mu^{\ast t}\in\mathcal{P}_+$ for all $t>0$. For example, Poisson law $Po(\lambda)$ and Gamma distribution are regular. It is known that if $\mu$ is regular, then $\Lambda(\mu)$ is free regular (see \cite{AHS13}). Therefore we get the following formula.

\begin{proposition}
For a regular distribution $\mu$, there exists a unique $\sigma\in\mathcal{P}_+$ such that 
$$\Psi(\mu)([0,\cdot])=\exp\left[1-\frac{1}{\Phi(\sigma)([0,\cdot])}\right].
$$
\end{proposition}
\begin{proof}
If $\mu$ is a Dirac measure, then we take $\sigma=\mu$. We may assume that $\mu$ is not a Dirac measure. For a regular distribution $\mu\in\mathcal{P}_+$, the measure $\Lambda(\mu)$ is free regular. By \cite{AHS13}, there exists a unique $\sigma\in\mathcal{P}_+$ such that $\Lambda(\mu)=B_1(\sigma)$. Then
\begin{align*}
\Psi(\mu)([0,\cdot])&=\mathcal{X}^{\lor}\circ \Phi \circ B_1^{-1} (\Lambda(\mu))([0,\cdot])\\
&=\mathcal{X}^{\lor}\circ \Phi \circ B_1^{-1} (B_1(\sigma))([0,\cdot])\\
&=\mathcal{X}^{\lor}\circ \Phi (\sigma)([0,\cdot])\\
&=\exp\left[1-\frac{1}{\Phi(\sigma)([0,\cdot])}\right].
\end{align*}
\end{proof}

%%%%%%%%%%%%%%%%%%%%%%%%
%Limit theorems for extreme values
%%%%%%%%%%%%%%%%%%%%%%%%

\section{Limit theorems for extreme values}

%%%%%%%%%%%%%%%%%%%%%%%%%%%%%%%
%FREE EXTREMES AND MARCHENKO-PASTUR LAWS
%%%%%%%%%%%%%%%%%%%%%%%%%%%%%%%

\subsection{Extreme values and Marchenko-Pastur laws}
Let $(\mathcal{M},\tau)$ be a tracial $W^*$-probability space. Suppose that $\{U_n\}_n$ is a sequence of freely independent identically distributed (bounded and positive) random variables in $(\mathcal{M},\tau)$ and $U_1\sim U(0,1)$. Define
\begin{align*}
\tilde{U}_n:=U_1\lor \cdots \lor U_n,\qquad n\in\mathbb{N}.
\end{align*}
Then we have
\begin{align*}
\tau(E_{\tilde{U}_n}([0,x]))&=\max \{n \tau(E_{U_1}([0,x]))-(n-1),0 \}\\
&=\begin{cases}
0, & 0 \le x <1-1/n\\
nx-(n-1) & 1-1/n\le x <1\\
1, & x\ge 1.
\end{cases}
\end{align*}

\begin{proposition}\label{prop:unif}
Let $F$ be a distribution function on $\mathbb{R}$. Then
\begin{align*}
\tau(E_{\tilde{U}_n}([0,F^{1/n}]))\xrightarrow{w} \Lambda^{\lor}(F), \qquad n\rightarrow\infty.
\end{align*}
In particular, we have
\begin{align*}
\tau(E_{\tilde{U}_n}([0,C_{\text{I}}^{1/n}]))&\xrightarrow{w}  F_{\text{I}},\\
\tau(E_{\tilde{U}_n}([0,C_{\text{II},\alpha}^{1/n}]))&\xrightarrow{w}  F_{\text{II},\alpha},\\
\tau(E_{\tilde{U}_n}([0,C_{\text{III},\alpha}^{1/n}]))&\xrightarrow{w}  F_{\text{III},\alpha},
\end{align*}
as $n\rightarrow\infty$ and $\alpha>0$.
\end{proposition}
\begin{proof}
Denote by $C(F)$ the set of all points of continuity of $F$. For all $x\in C(F)$, we have
\begin{align*}
\tau(E_{\tilde{U}_n}([0,F(x)^{1/n}]))&=\begin{cases}
0, & 0 \le F(x) <(1-1/n)^n\\
nF(x)^{1/n}-(n-1) & (1-1/n)^n\le F(x) <1\\
1, & F(x)\ge 1.
\end{cases}\\
&\xrightarrow{n\rightarrow\infty}\begin{cases}
0, & 0\le F(x)<e^{-1}\\
1+\log F(x), & e^{-1}\le F(x)<1\\
1, & F(x)\ge1.
\end{cases}\\
&=\Lambda^{\lor}(F(x)).
\end{align*}
The last assersion is clear since $\Lambda^{\lor}$ maps the classical extreme values to the corresponding type free extreme values.
\end{proof}

Recall that $\Phi(\pi)$ is the uniform distribution on $(0,1)$. Hence we get the following convergence.

\begin{corollary}\label{cor:freePoisson}
For any distribution functions $F$ on $\mathbb{R}$,
\begin{align*}
\Phi(\pi^{\boxplus n})([0,nF^{1/n}])\xrightarrow{w} \Lambda^{\lor}(F), \qquad n\rightarrow\infty.
\end{align*}
\end{corollary}
\begin{proof}
By Theorem \ref{thm:free_additive_max}, we have 
\begin{align*}
\Phi(\pi^{\boxplus n})([0,nF^{1/n}])=\Phi(D_{1/n}(\pi^{\boxplus n}))([0,F^{1/n}])=\Phi(\pi)^{\Box \hspace{-.55em} \lor n}([0,F^{1/n}]).
\end{align*}
By the above discussion and by Proposition \ref{prop:unif}, we can obtain the above convergence.
\end{proof}

Next we suppose that $\{V_n\}_n$ is a sequence of Boolean independent identically distributed (bounded and positive) random variables in $(\mathcal{M},\tau)$ and $V_1\sim U(0,1)$. Define
\begin{align*}
\tilde{V}_n:=V_1\lor \cdots \lor V_n, \qquad n\in\mathbb{N}.
\end{align*}
Then we have
\begin{align*}
\tau(E_{\tilde{V}_n}([0,x]))&=\frac{\tau(E_{V_1}([0,x]))}{n-(n-1)\tau(E_{V_1}([0,x]))}\\
&=\begin{cases}
1, & x\ge 1\\
\frac{x}{n-(n-1)x} & 0\le x < 1.
\end{cases}
\end{align*}

\begin{proposition}\label{prop:unifboolean}
Let $F$ be a distribution function on $\mathbb{R}$. Then
\begin{align*}
\tau(E_{\tilde{V}_n}([0,F^{1/n}]))\xrightarrow{w} (\mathcal{X}^\lor)^{-1} (F), \qquad n\rightarrow\infty.
\end{align*}
In particular, we have $\tau(E_{\tilde{V}_n}([0,C_{\text{II},\alpha}^{1/n}]))\xrightarrow{w} B_{\text{II},\alpha}$ as $n\rightarrow\infty$ and $\alpha>0$.
\end{proposition}
\begin{proof}
For any distribution functions $F$ on $\mathbb{R}$,
\begin{align*}
\tau(E_{\tilde{V}_n}([0,F^{1/n}]))&=\frac{F^{1/n}}{n-(n-1)F^{1/n}}\\
&=\frac{1}{1-n(1-F^{-1/n})}\\
&\xrightarrow{w} \frac{1}{1-\log F}=(\mathcal{X}^\lor)^{-1} (F).
\end{align*}
The last assersion is clear since $(\mathcal{X}^\lor)^{-1}(C_{\text{II},\alpha})=B_{\text{II},\alpha}$.
\end{proof}

As the same argument of Corollary \ref{cor:freePoisson}, we get the following convergence.

\begin{corollary}
For any distribution functions $F$ on $\mathbb{R}$,
\begin{align*}
\Phi(\pi^{\uplus n})([0,nF^{1/n}])\xrightarrow{w} (\mathcal{X}^{\lor})^{-1}(F), \qquad n\rightarrow\infty.
\end{align*}
\end{corollary}
\begin{proof}
By Theorem \ref{thm:Boolean_additive_max}, we have 
\begin{align*}
\Phi(\pi^{\uplus n})([0,nF^{1/n}])=\Phi(D_{1/n}(\pi^{\uplus n}))([0,F^{1/n}])=\Phi(\pi)^{\cup \hspace{-.55em} \lor n}([0,F^{1/n}]).
\end{align*}
By the above discussion and by Proposition \ref{prop:unifboolean}, we can obtain the above convergence.
\end{proof}

%%%%%%%%%%%%%%%%%%%%%%%%%%%%%%%%%%%%%%
%EXTREME VALUES AND FREE MULTIPLICATIVE CONVOLUTION
%%%%%%%%%%%%%%%%%%%%%%%%%%%%%%%%%%%%%%

\subsection{Extreme values and free multiplicative convolution}
In this section, we obtain a relation between free/Boolean extreme values and free multiplicative convolution.

\begin{proposition}\label{prop:freemultiplicative}
For $n\in\mathbb{N}$ and $\mu\in\mathcal{P}_+$, we have
\begin{align}\label{boxtimes}
\Phi(\mu^{\boxtimes n})([0,x^n])=\Phi(\mu)([0,x]).
\end{align}
\end{proposition}
\begin{proof}
If $\mu$ is a Dirac measure $\delta_a$ for some $a\ge 0$, then $\mu^{\boxtimes n}=\delta_{a^n}$. Therefore we get the equation \eqref{boxtimes}. Therefore we may assume that $\mu$ is not a Dirac measure. Note that $\mu^{\boxtimes n}(\{0\})=\mu(\{0\})$ for all $n\in\mathbb{N}$. By Proposition \ref{prop:HMlimit} and a property of $S$-transform, we have
\begin{align}\label{eq:free multi}
z=\Phi(\mu^{\boxtimes n})\left(\left[ 0,\frac{1}{S_{\mu^{\boxtimes n}}(z-1)}\right] \right)=\Phi(\mu^{\boxtimes n})\left( \left[ 0,\frac{1}{S_\mu(z-1)^n}\right]\right)
\end{align}
for all $z\in (\mu(\{0\}),1)$. Take $z=S_\mu^{-1}(1/x)+1\in (\mu(\{0\}),1)$ in \eqref{eq:free multi}. Then
\begin{align}\label{eq:S-multi}
\Phi(\mu^{\boxtimes n})\left( \left[ 0,\frac{1}{S_\mu(S_\mu^{-1}(1/x)+1-1)^n}\right]\right)=S_\mu^{-1}\left(\frac{1}{x}\right)+1.
\end{align}
Since $S_\mu^{-1}(1/x)+1=\Phi(\mu)([0,x])$ by Lemma \ref{lem:pS}, the equation \eqref{eq:S-multi} holds if and only if \eqref{boxtimes} does.
\end{proof}

Take $\mu=f_\alpha^+$ or $\mu=b_\alpha^+$ in \eqref{boxtimes}. Then we get the following formula.
\begin{corollary}\label{cor:FB}
For $\alpha\in(0,1)$ and $n\in\mathbb{N}$, we have
\begin{align*}
\Phi({f_\alpha^+}^{\boxtimes n})([0,x^n])&=F_{\text{II}, \frac{\alpha}{1-\alpha}}(x).\\
\Phi({b_\alpha^+}^{\boxtimes n})([0,x^n])&=B_{\text{II}, \frac{\alpha}{1-\alpha}}(x).
\end{align*}
\end{corollary}

\begin{remark}
By using a relation in \cite{AH16}:
\begin{align*}
\pi^{\frac{1-\alpha}{\alpha}}\boxtimes f_{\alpha}^+=b_{\alpha}^+, \qquad \alpha\in(0,1),
\end{align*}
we have 
\begin{align*}
\Phi(\pi\boxtimes f_{1/2}^+)([0,\cdot])=\Phi(b_{1/2}^+)([0,\cdot])=B_{\text{II},1}(\cdot),
\end{align*}
where the last equation holds when we take $n=1$ and $\alpha=1/2$ in Corollary \ref{cor:FB}.
\end{remark}

\subsection*{Acknowledgment}
The author would like to express hearty thanks to Prof. Takahiro Hasebe (Hokkaido University) for his precious advices. This research is an outcome of Joint Seminar supported by JSPS and CNRS under the Japan-France Research Cooperative Program.

%%%

\vspace{0.6cm}
\hspace{-5.5mm}{\it Yuki Ueda\\
Department of Mathematics, Hokkaido University,\\
Kita 10, Nishi 8, Kita-Ku, Sapporo, Hokkaido, 060-0810, Japan\\
email: yuuki1114@math.sci.hokudai.ac.jp}

\end{document}